\documentclass[11pt, a4paper]{amsart}

\usepackage[ansinew]{inputenc}
\usepackage[all]{xy}
\SelectTips{cm}{}
\usepackage{hyperref}
\hypersetup{pdfauthor={Ludovico Battista, Stefano Francaviglia, Marco Moraschini, Filippo Sarti, Alessio Savini}
  ,pdftitle=Bounded cohomology and differential forms}

\usepackage{amsfonts}
\usepackage{amssymb, amsmath,amsthm, mathtools}
\usepackage{tabularx, hyperref}
\usepackage{enumerate}
\usepackage{nicefrac}
\usepackage{tikz-cd}

\usepackage{stmaryrd}

\makeatletter
\newtheorem*{rep@theorem}{\rep@title}
\newcommand{\newreptheorem}[2]{%
\newenvironment{rep#1}[1]{%
 \def\rep@title{#2 \ref{##1}}%
 \begin{rep@theorem}}%
 {\end{rep@theorem}}}
\makeatother

\newtheorem{intro_thm}{Theorem}
\newreptheorem{intro_thm}{Theorem}  

\newtheorem*{thm*}{Theorem}

\newtheorem{lemma}{Lemma}[section]
\newtheorem{thm}[lemma]{Theorem}
\newreptheorem{thm}{Theorem}  
\newtheorem{prop}[lemma]{Proposition}
\newreptheorem{prop}{Proposition}  

\newreptheorem{cor}{Corollary}  

\theoremstyle{definition}

\newreptheorem{quest}{Question}  

\newtheorem{example}[lemma]{Example}
\newtheorem{rem}[lemma]{Remark}

\theoremstyle{definition}

\makeatletter
\newcommand\norm{\bBigg@{0.8}}

\makeatother
 
 \newcommand{\indnorm}[2][flex]{\csname #1l\endcsname\|#2%
                                 \csname #1r\endcsname\|\mathclose{}}
                                  \newcommand{\indnorml}[4][flex]{\csname #1l\endcsname\|#2%
                                 \csname #1r\endcsname\|_{#3}^{#4}\mathclose{}}
\newcommand{\sv}[2][flex]{\indnorm[#1]{#2}}

\newcommand{\genrel}[3][flex]{\csname #1l\endcsname\langle #2 \mathbin{\csname #1m\endcsname|} #3\csname #1r\endcsname\rangle}

\DeclareMathOperator{\st}{\textup{st}}

\DeclarePairedDelimiter{\tonde}{(}{)}

\newcommand{\R} {\ensuremath {\mathbb{R}}}

\newcommand{\clforms} {C\Omega}

\def\phi{\varphi}

\usepackage{color}
\usepackage{pdfcolmk}

\long\def\forget#1{}

\makeatletter
\@namedef{subjclassname@2020}{%
  \textup{2020} Mathematics Subject Classification}
\makeatother

\begin{document}

\title[Bounded cohomology classes of exact forms]{Bounded cohomology classes of exact forms}

\author[]{Ludovico Battista}
\address{Fondazione Bruno Kessler}
\email{lbattista@fbk.eu}

\author[]{Stefano Francaviglia}
\address{Dipartimento di Matematica, Universit\`{a} di Bologna, Bologna, Italy}
\email{stefano.francaviglia@unibo.it}

\author[]{Marco Moraschini}
\address{Dipartimento di Matematica, Universit\`{a} di Bologna, Bologna, Italy}
\email{marco.moraschini2@unibo.it}

\author[]{Filippo Sarti}
\address{Dipartimento di Matematica, Universit\`{a} di Torino, Torino, Italy}
\email{filippo.sarti@unito.it}

\author[]{Alessio Savini}
\email{alessio.xyz@gmail.com}

\thanks{}

\date{\today.\ }

\begin{abstract}
On negatively curved compact manifolds, it is possible to associate to every closed form a
bounded cocycle -- hence a bounded cohomology class -- via integration over straight
simplices. The kernel of this map is contained in the space of exact forms. We
show that in degree 2 this kernel is trivial, in contrast with higher degree. 
In other words, exact non-zero $2$-forms define non-trivial bounded cohomology classes. 

This result is the higher dimensional version of a classical theorem by Barge and Ghys for
surfaces~\cite{BargeGhys}. As a consequence, one gets that the second bounded cohomology of
negatively curved manifolds contains an infinite dimensional space, whose classes are explicitly
described by integration of forms. This also showcases that some recent results by Marasco~\cite{marasco2022cup, marasco2022massey} can be applied 
in higher dimension to obtain new non-trivial results on the vanishing of certain cup products
and Massey products. Some other applications are discussed.
\end{abstract}

\maketitle

\section{Introduction}
Let $M$ be a smooth manifold. Let $\Omega^2(M)$ denote the
space of differential $2$-forms on $M$ 
and let $C\Omega^2(M)$ and $E\Omega^2(M)$ be the subspaces of closed and exact forms, respectively.

Barge and Ghys in the late 80s showed that the second bounded cohomology group of a negatively curved closed surface contains an infinite dimensional subspace given by the space of differential $2$-forms:
\begin{thm*}[{\cite{BargeGhys}}]
Let $\Sigma$ be an oriented closed connected negatively curved surface. Then, there exists an embedding
\[
\Psi \colon \Omega^2(\Sigma) \to H^2_b(\Sigma; \R),
\]
where $H^2_b(\Sigma; \R)$ denotes the second bounded cohomology group of $\Sigma$.
\end{thm*}
This classical theorem has attracted new attention after the recent results by Marasco on the vanishing of certain cup products and Massey products:
\begin{thm*}[\cite{marasco2022cup, marasco2022massey}]
Let $M$ be an oriented closed connected negatively curved manifold with (possibly empty) convex boundary. Let $\omega \in E\Omega^2(M)$ be an exact form and $\alpha \in H^k_b(M; \R)$, then the cup product
\[
\Psi(\omega) \cup \alpha = 0 \in H_b^{k+2}(M; \R),
\]
where $\Psi$ is the Barge--Ghys straightening morphism (Section~\ref{subsec:def:BG:morph}). Moreover, if $\omega \in E\Omega^2(M), \alpha_1 \in H^{k_1}_b(M; \R)$ and $\alpha_2 \in H^{k_2}_b(M; \R)$ with $k_1, k_2 \geq 1$, then also the triple Massey product $\langle \alpha_1, \Psi(\omega), \alpha_2 \rangle $ in $H^{k_1 + k_2 +1}_b(M; \R)$ vanishes.
\end{thm*}

In this short note we observe that the proof of Barge--Ghys' Theorem extends to higher dimensional
manifolds, getting:
\begin{intro_thm}\label{main:thm}
 Let $M$ be an oriented closed connected negatively curved manifold. 
  Then, the Barge--Ghys
straightening morphism $$\Psi \colon C\Omega^2(M) \to H^2_b(M; \R)$$
is injective. In particular, both $\Psi(E\Omega^2(M))$ and $\Psi(C\Omega^2(M))$ are infinite
dimensional subspaces of $H_b^2(M;\R)$.
\end{intro_thm}

This result showcases that Marasco's result provides non-trivial information for closed
manifolds of any dimension. For instance we have:
\begin{example}
Let $n \geq 2$ and $M$ be an oriented closed connected negatively curved
 $n$-manifold. Since $M$ is negatively curved, we can pick a non-trivial element $\alpha \in H_b^n(M; \R)$ (e.g.\ the volume form~\cite{vbc, iy82}). By Theorem~\ref{main:thm}, for every non-trivial $\omega \in E\Omega^2(M)$ the class $\Psi(\omega) \in H_b^2(M; \R)$ is non-zero. Hence we have
\[ \Psi(\omega) \cup \alpha = 0 \in H_b^{n+2}(M; \R), \]
where both classes $\Psi(\omega)$ and $\alpha$ are non-trivial. This
can be interpreted as the first non-trivial vanishing result for cup products of bounded \emph{geometric} classes of arbitrary dimension.
\end{example}

As an example of other kind of consequences of Theorem~\ref{main:thm}, we have:
\begin{intro_thm}\label{cor:surj:still:inj}
Let $M$ be a manifold and let $N$ be an oriented closed connected negatively curved manifold. Suppose that there exists a continuous map $f \colon M \to N$ that induces a surjective homomorphism at the level of fundamental groups. Then, there exists a natural embedding:
\[
C\Omega^2(N) \hookrightarrow H^2_b(M; \R).
\]
\end{intro_thm}

Other corollaries of Theorem~\ref{main:thm} are discussed in Section~\ref{sec:generalise:BG}
(see for instance Examples~\ref{ex:fiber} and~\ref{ex:arithmetic}). 
Moreover, in Section~\ref{sec:totally:geodesic} we prove that such kind of results can be applied successfully also to the case of
totally geodesic boundary (see also Proposition~\ref{prop:totally:geodesic:submfld}):

\begin{intro_thm}\label{thm:totally:geodesic}
Let $n \geq 3$ and let $M$ be an oriented compact connected negatively curved $n$-manifold with convex boundary.
Also suppose that at least one connected boundary component is totally geodesic. Then, $\Psi(C\Omega^2(M))$ and $\Psi(E\Omega^2(M))$ are  infinite dimensional.
\end{intro_thm}
On the other hand, we notice that the Barge--Ghys straightening morphism can be trivial for
general negatively curved manifold with convex non-empty boundary
(see Section~\ref{sec:noninj}).

\subsection*{Acknowledgements}
We would like to thank Domenico Marasco for useful conversations
and Alan Reid for suggesting us Example~\ref{ex:arithmetic}. This work originated from
the weekly meetings of the topology group of the University of Bologna, and during the workshop
\emph{Recent advances in bounded cohomology} at \emph{Universit\"at Regensburg}.

The authors were partially supported by GNSAGA group of INdAM, and by PRIN 2017JZ2SW5.
Moreover, M.\ M. and F.\ S. have been supported by the INdAM -- GNSAGA Project, CUP E55F22000270001.

\section{Basic definitions and notations}\label{sec:prel}
In this section we recall the main definitions that we need in the paper.
\subsection{Bounded cohomology} Let $X$ be a topological space. We denote by $(C^\bullet(X; \R), \delta^\bullet)$ the standard real singular cochain complex. Gromov defined the \emph{bounded cohomology} of spaces as follows~\cite{vbc}: given a singular cochain $\varphi \in C^k(X; \R)$, the $\ell^\infty$-\emph{norm} of $\varphi$ is 
\[
\sv{\varphi}_\infty \coloneqq \sup\{|\varphi(\sigma)| \, , \, \sigma \mbox{ is a singular $k$-simplex}\}.
\]
We denote by $C_b^\bullet(X; \R) \subseteq C^\bullet(X; \R)$ the subspace of \emph{bounded cochains}, i.e.\ those cochains such that $\sv{\varphi}_\infty < +\infty$. Since the standard coboundary operator sends bounded cochains to bounded cochains, we have that $(C_b^\bullet(X; \R), \delta^\bullet)$ is a cochain complex. 

The \emph{bounded cohomology of} $X$ (with real coefficients)
is then:
\[
H_b^*(X; \R) \coloneqq H^*(C^\bullet_b(X; \R), \delta^\bullet).
\]
Similarly, since bounded cohomology is a homotopy invariant~\cite{vbc, Ivanov17}, we can define the (real) bounded cohomology of a group $\Gamma$ simply as 
\[
H^\bullet_b(\Gamma; \R) \coloneqq H^\bullet_b(B\Gamma; \R),
\]
for any model $B\Gamma$.
\subsection{The Barge--Ghys straightening morphism}\label{subsec:def:BG:morph}
Let $M$ be an oriented compact connected negatively curved manifold $M$ with (possibly empty)
convex boundary. 

For every singular simplex  $\sigma \colon \Delta\to M$ we can define
its straightening by lifting $\sigma$ to $\widetilde \sigma:\Delta\to \widetilde M$, then
pulling it tight relatively to vertices, and projecting back to $M$ by convexity. In case of hyperbolic
manifolds, this coincides with (a suitable parametrization of) the geodesic convex-hull of the vertices of
$\widetilde\sigma$. In general the construction is slightly different, and it is done by
recursively coning from a vertex to opposite faces~\cite[Section~$8.4$]{Frigerio:book}.

This construction defines a map $\st \colon C_*(M; \R) \to C_*(M; \R)$ which is homotopic to
the identity~\cite{vbc, Frigerio:book, marasco2022cup}. 
We can therefore associate to every $\omega \in C\Omega^2(M)$ a $2$-cocycle $c_\omega$ defined as
\[ c_\omega (\sigma) = \int_{\st(\sigma)} \omega. \]
The following is classical:

\begin{lemma}[{\cite[Lemma 3.1]{BargeGhys}, \cite[Section~2.2]{marasco2022cup}}]\label{lemma1}
    The 2-cocycle $c_\omega$ is bounded.
\end{lemma}
The proof of Lemma~\ref{lemma1} goes as follows. Since $M$ is compact, $\omega$ is
bounded on orthonormal 2-frames. In the hyperbolic case then the claim follows from the well-known
bound on areas of geodesic triangles; the same being true in negative curvature setting,
where the bound now depends on curvature-bounds~\cite{iy82}.

Lemma~\ref{lemma1} shows that there exists a well-defined cochain map
\[ \psi \colon \clforms^2(M) \to Z^2_b(M; \R), \qquad \omega \mapsto c_\omega \]
from the space of closed $2$-forms to the space of bounded $2$-cocycles of $M$.
By passing to (bounded) cohomology on the right, we obtain the \emph{Barge--Ghys straightening morphism}
\[ \Psi \colon \clforms^2(M) \to H^2_b(M; \R). \]


\section{Barge--Ghys Theorem for closed negatively curved manifolds}\label{sec:generalise:BG}

In this section we show how to prove Theorem~\ref{main:thm} by combining the original proof by Barge--Ghys with the following result of the late 90s:

\begin{thm}[{\cite[Corollary 1.5]{CrokeShara}}] \label{zeroongeodesics}
  Let $M$ be a closed negatively curved manifold, $\alpha$ a (not
a priori closed) smooth $1$-form on $M$, and $\beta$ a closed $1$-form on $M$. If for every closed geodesic $\gamma$ in
$M$ we have $$\int_\gamma\alpha=\int_\gamma\beta$$ 
then $\alpha$ is closed and $[\alpha]=[\beta]$ in $H^1(M;\R)$. 
\end{thm}





This result was only known for $n=2$ when Barge and Ghys wrote their paper, but now one can use it to extend their argument to the case of oriented closed connected negatively curved $n$-manifolds. More precisely, we have the following:

\begin{repthm}{main:thm}
 Let $M$ be an oriented closed connected negatively curved manifold. 
  Then, the Barge--Ghys
straightening morphism $$\Psi \colon C\Omega^2(M) \to H^2_b(M; \R)$$
is injective. In particular, $\Psi(E\Omega^2(M))$ and $\Psi(C\Omega^2(M))$ are infinite
dimensional subspaces of $H_b^2(M;\R)$.
\end{repthm}

For the rest of the section we assume that $M$ is an oriented closed connected negatively curved manifold. Following Barge--Ghys' proof, we first show that if $\Psi(\omega) = 0$, then $\omega$ has to be exact~\cite[Lemma 3.4]{BargeGhys}:




\begin{lemma}\label{lemma:omega:exact}
    Let $\omega \in C\Omega^2(M)$ be such that $\Psi(\omega) = 0$. Then, $\omega$ is exact.
\end{lemma}
\begin{proof}
    Recall that we have the following commutative diagram~\cite[Section~2.2]{marasco2022cup}:
    \[
    \xymatrix{
    C\Omega^2(M; \R) \ar[rr] \ar[d]^-{\Psi} && H^2_{\textup{dR}}(M; \R) \ar[d]^-{\cong} \\
    H^2_b(M; \R) \ar[rr] && H^2(M; \R),
    }
    \]
    where $H^2_{\textup{dR}}(M; \R)$ denotes the de Rham cohomology of $M$ and the lower horizontal arrow is the comparison map~\cite{vbc}. Hence, if $\Psi(\omega) = 0$, we have that $\omega$ is mapped to the zero element in $H_{\textup{dR}}^2(M; \R)$. This shows that $\omega$ has to be exact.
\end{proof}

    \begin{rem}\label{remark_quasi}
    One standard way to prove that $H^2_b(M; \R)$ does not vanish is to show that the kernel of the comparison map 
    \[ \textup{comp}^2 \colon H^2_b(M; \R) \to H^2(M; \R) \]
    is non-trivial. In the case of bounded cohomology of groups, the kernel of $\textup{comp}^2$, called \emph{second exact bounded cohomology} and denoted by $EH^2_b(M; \R)$, is usually studied via (homogeneous) quasi-morphisms~\cite[Proposition 2.8, Corollary 2.11]{Frigerio:book}. In this setting, the construction by Barge and Ghys leads to the so-called \emph{de Rham quasi-morphisms}~\cite[Section 2.3.1]{calegari}. 
      From the diagram in the proof of Lemma \ref{lemma:omega:exact}, it follows that $\Psi(E\Omega^2(M)) \subseteq EH^2_b(M; \R)$. A corollary of Theorem~\ref{main:thm} is  therefore that non-zero
      exact forms are non-trivial elements in the second exact bounded cohomology group of $M$.
    \end{rem}

    %

We assume now that $\omega$ is an exact $2$-form such that $\Psi(\omega) = 0$. This allows us to fix the following notations: $\omega = d \alpha$ and $c_\omega = \delta \tau$, where $\alpha$ and $\tau$ are a differential $1$-form and a bounded 1-cochain, respectively. 

\begin{lemma}\label{lemma2}
    Let $\gamma_1, \ldots, \gamma_s$ be closed oriented geodesics. Fix a base-point on each $\gamma_i$ and a parametrization by constant speed in order to consider each $\gamma_i$ as a 1-cycle.
    If $\sum_{i=1}^n \gamma_i = 0 \in H_1(M; \R)$, then
    \[ \sum_{i=1}^n \int_{\gamma_i} \alpha = \sum_{i=1}^n \tau (\gamma_i). \]
\end{lemma}

\begin{proof} In the case of surfaces, this is exactly~\cite[Lemma 3.5]{BargeGhys}.
    Let $S$ be a 2-chain such that $\partial c = \sum_{i=1}^n \gamma_i$. Using the properties of $\st$, we can straighten each 2-simplex of $S$ without changing the 1-skeleton and suppose that $S$ is made by straight simplices. Then:
    
    $$
      \sum_{i=1}^n \int_{\gamma_i} \alpha= \int_{\partial S} \alpha = \int_S d\alpha
      =\int_S\omega=c_\omega(S)=\delta\tau(S)=\tau(\partial S)=
                                           \sum_{i=1}^n \tau (\gamma_i).$$
\end{proof}

\begin{lemma}[{\cite[Lemma 3.6]{BargeGhys}}]\label{lemma3}
    Let $\gamma_1, \ldots, \gamma_s$ be closed oriented geodesics such that $\sum_{i=1}^n \gamma_i = 0 \in H_1(M; \R)$. Then
    \[ \sum_{i=1}^n \int_{\gamma_i} \alpha = 0. \]
\end{lemma}

For the convenience of the reader, we recall the proof by Barge and Ghys:

\begin{proof} Let us denote by $\gamma_i^N$ the $N$-fold concatenation of the geodesic $\gamma_i$ with itself, considered as a single singular $1$-simplex. Since the element $ \sum_{i=1}^n \gamma_i^N$ is clearly null-homologous, by Lemma~\ref{lemma2} we have that for every $N \in \mathbb{N}$:
    \[ N \tonde*{\sum_{i=1}^n \int_{\gamma_i} \alpha } = \sum_{i=1}^n \int_{\gamma_i^N} \alpha = \tau \tonde*{ \sum_{i=1}^n \gamma_i^N }.  \]
Since $\tau$ is bounded, the last term is bounded independently of $N$.
But $N$ is arbitrary, and this concludes the proof.
  \end{proof}

All the previous results lead to the following important lemma  (compare with Barge--Ghys~\cite[Lemma 3.7]{BargeGhys}):

\begin{lemma}\label{lemma4}
    There exists a closed 1-form $\beta$ such that for every closed geodesic $\gamma$
    \[ \int_\gamma \alpha = \int_\gamma \beta. \]
\end{lemma}

\begin{proof}
    By Lemma~\ref{lemma3}, if $\gamma_1$ and $\gamma_2$ are homologous, then
    \[ \int_{\gamma_1} \alpha = \int_{\gamma_2} \alpha. \]
    Recall that in a negatively curved manifold each non-trivial free homotopy class of closed curves (and hence
    each homology class) contains a geodesic representative \cite[Theorem 3.8.14]{klingenberg2011riemannian}. Then by Lemma \ref{lemma3} the integration of $\alpha$ over closed geodesics defines a map $I$ from $H_1(M; \R)$ to $\mathbb{R}$. This map is linear: if $\gamma_1 + \gamma_2$ is homologous to $\gamma_3$, then Lemma~\ref{lemma3} shows that $I(\gamma_3)=I(\gamma_1)+I(\gamma_2).$ Hence, the map $I$ defines an element in $H^1(M; \R) \cong \textup{Hom}(H_1(M; \R); \R)$. By de Rham Isomorphism Theorem there exists a closed 1-form $\beta$ whose integral represent such class. This shows that 
    \[ \int_\gamma \alpha = \int_\gamma \beta, \]
    for every closed geodesic $\gamma$.
\end{proof}

Theorem \ref{main:thm} now immediately follows.

\begin{proof}[Proof of Theorem \ref{main:thm}]
    By Lemma~\ref{lemma4} and Theorem \ref{zeroongeodesics},  the form $\alpha$ is closed, whence $\omega= d\alpha = 0$.
\end{proof}

Combining Theorem~\ref{main:thm} with classical facts on bounded cohomology we get the following result:

\begin{repthm}{cor:surj:still:inj}
Let $M$ be a manifold and let $N$ be an oriented closed connected negatively curved manifold. Suppose that there exists a continuous map $f \colon M \to N$ that induces a surjective homomorphism at the level of fundamental groups. Then, there exists a natural embedding:
\[
C\Omega^2(N) \hookrightarrow H^2_b(M; \R).
\]
\end{repthm}
\begin{proof}
It is well known that every surjective group homomorphism induces an injective map on second bounded cohomology groups~\cite{Bouarich:exact}. Hence, we have that the surjective group homomorphism
\[
\pi_1(f) \colon \pi_1(M) \to \pi_1(N)
\]
induces an injective map between the second bounded cohomology groups:
\[
H^2_b(\pi_1(f)) \colon H^2_b(\pi_1(N); \R) \to H^2_b(\pi_1(M); \R).
\]
Moreover, by Gromov's Mapping Theorem~\cite{vbc, Ivanov17} (and the homotopy invariance of bounded cohomology) we have the following commutative diagram:
\[
\xymatrix{
H^2_b(\pi_1(N); \R) \ar[rr]^-{H^2_b(\pi_1(f))} \ar[d]^-{\cong} && H^2_b(\pi_1(M); \R) \ar[d]^-{\cong} \\
H^2_b(N; \R) \ar[rr]^-{H^2_b(f)} && H^2_b(M; \R)
}
\]
where the vertical lines are (isometric) isomorphisms. Since $H^2_b(\pi_1(f))$ is injective, the commutativity of the diagram shows that also $H^2_b(f)$ is injective.
The thesis now follows by considering the following composition
\[
C\Omega^2(N) \xrightarrow{\Psi} H^2_b(N; \R) \xrightarrow{H^2_b(f)} H^2_b(M; \R),
\]
where $\Psi$ is injective because of Theorem~\ref{main:thm}.
\end{proof}

Before discussing some applications of the previous result, let us briefly recall what
is known in the literature about the bounded cohomology of manifolds (and groups).
Computing the bounded cohomology of manifolds with non-amenable fundamental group 
is usually a hard task. Indeed, we do not know any example of a closed aspherical manifold
with non-amenable fundamental group such that its bounded cohomology is fully understood.
Nevertheless, bounded cohomology in degree $2$ and degree $3$ is fairly well-studied
using the technology of quasi-morphisms to construct non-trivial bounded cohomology
classes (Remark \ref{remark_quasi}). This approach shows that the second bounded cohomology group of a manifold
with acylindrically hyperbolic group (e.g.\ a negatively curved manifold) is infinite dimensional~\cite{Bro81, EF97,
Fuj98, BF02, Fuj00}. In this situation, with the help of hyperbolic geometry~\cite{Soma, Somb,FujiSoma}
and more delicate constructions with quasi-cocycles~\cite{fps,ffps}, one can also prove that the third bounded 
cohomology group of a manifold with acylindrically hyperbolic group is infinite dimensional.
However, it is still not known, e.g., whether surfaces or hyperbolic $3$-manifolds have
trivial fourth bounded cohomology or not.

Theorem~\ref{cor:surj:still:inj} provides a new \emph{geometric} way to describe infinite-dimensional subspaces of the second bounded cohomology group of many non-positively curved manifolds:

\begin{example}\label{ex:fiber}
Let $M$ be a fiber bundle over an oriented closed connected negatively curved manifold $N$ with connected fiber, then Theorem~\ref{cor:surj:still:inj} implies that
\[
C\Omega^2(N) \hookrightarrow H^2_b(M; \R).
\]
For instance, this applies when $M$ is the product of oriented closed connected negatively curved manifolds. 
\end{example}

Another interesting source of applications of Theorem~\ref{cor:surj:still:inj} arises from complete finite-volume arithmetic hyperbolic manifolds of simplest type~\cite[Section 4]{kolpakovreidslavich}.
Recall that all complete finite-volume arithmetic hyperbolic manifolds of simplest type contain totally geodesic immersed complete finite-volume $k$-submanifolds 
for every $1 \leq k \leq n-1$~\cite[Section 2.1]{fisherstover}. 
This implies that if $M$ is a complete finite-volume arithmetic hyperbolic $n$-manifold of simplest type and $M'$ is a finite covering of $M$,
then for every $k \in \{1, \cdots, n-1\}$ there exists an immersed totally geodesic complete finite-volume $k$-submanifold $S$ of $M'$ and so the inclusion 
of $S$ into $M'$ is $\pi_1$-injective.  Hence, we have that
$\pi_1(S) \leq \pi_1(M') \leq \pi_1(M)$.
Using a result by Bergeron, Haglund and Wise about \emph{closed} arithmetic hyperbolic manifolds, we obtain the following:

\begin{example}\label{ex:arithmetic}
Let $M$ be a closed arithmetic hyperbolic $n$-manifold of simplest type. By Bergeron, Haglund and Wise~\cite[Theorem 1.2]{bergeronhaglundwise}
there always exists a finite covering $M_1$ of $M$ with the following property: for every $k \in \{1, \cdots, n-1\}$
and every immersed totally geodesic $k$-submanifold $N \subseteq M_1$, there exists a finite covering $M_2$ of $M_1$ 
corresponding to a subgroup $\pi_1(M_2) \leq \pi_1(M)$ such that 
\begin{itemize}
\item $\pi_1(N) \leq \pi_1(M_2)$, and
\item  $\pi_1(M_2)$ \emph{algebraically retracts} onto $\pi_1(N)$, i.e.\ there exists a surjective homomorphism
$\rho \colon  \pi_1(M_2) \to \pi_1(N)$ such that $\rho$ is the identity over $\pi_1(N)$.
\end{itemize}
Since $N$ and $M_2$ are aspherical manifolds, the homomorphism $\rho$ is induced by a continuous map $r \colon M_2 \to N$. Hence $\pi_1(r) \colon \pi_1(M_2) \to \pi_1(N)$ is an algebraic retraction, whence a surjective homomorphism.
We can then apply Theorem~\ref{cor:surj:still:inj} to $r$ and conclude that 
\[
C\Omega^2(N) \hookrightarrow H^2_b(M_2; \R).
\]
\end{example}

\section{The case of totally geodesic boundary}\label{sec:totally:geodesic}
In this section we prove the following more general version of  Theorem~\ref{thm:totally:geodesic}:
\begin{prop}\label{prop:totally:geodesic:submfld}
 Let $M$ be a Riemannian manifold with (possibly non-empty) convex boundary. Suppose that $M$
 contains an oriented closed connected negatively curved totally geodesic submanifold
 $N$ (possibly contained in the boundary). Then, the images of $C\Omega^2(M)$ and $E\Omega^2(M)$ under the Barge--Ghys straightening morphism $\Psi \colon C\Omega^2(M) \to H^2_b(M; \R)$
 are infinite dimensional. 
\end{prop}
\begin{proof}
Let $i \colon N \to M$ denote inclusion of $N$ into $M$. We can then consider the following diagram:
\[
\xymatrix{
C\Omega^2(M) \ar[rr]^-{i^*} \ar[d]^-{\Psi_M} && C\Omega^2(N) \ar[d]^-{\Psi_N} \\
H^2_b(M; \R) \ar[rr]^-{H^2_b(i)} && H^2_b(N; \R).
}
\]
Note that $i^*$ sends exact forms to exact forms. Moreover, since $N$ admits a
tubular/collar neighborhood in $M$, exact forms of $N$ are restrictions of exact forms of $M$. Hence, the image of the map $i^*$ contains $E\Omega^2(N)$ and is infinite-dimensional.

Moreover,
Theorem~\ref{main:thm} shows that $\Psi_N$ is also injective. Hence, in order to prove the statement
it is sufficient to show that the previous diagram commutes. In fact, we are going to show that it commutes at the level of cochains, i.e.\ that the following diagram commutes:
\[
\xymatrix{
C\Omega^2(M) \ar[rr]^-{i^*} \ar[d]^-{\psi_M} && C\Omega^2(N) \ar[d]^-{\psi_N} \\
C^2_b(M; \R) \ar[rr]^-{C^2_b(i)} && C^2_b(N; \R).
}
\]
First, since $N$ is totally geodesic in $M$, the inclusion map $i$ sends straight simplices to straight simplices, that is
\[
i(\st(\sigma)) = \st(i(\sigma))
\]
for every $\sigma \colon \Delta^2 \to N$ singular $2$-simplex in $N$.
Then, on the one hand we have:
\begin{align*}
C^2_b(i) \circ \psi_M(\omega) &= C^2_b(i) (c_\omega) \\
&= \Big(\sigma \mapsto c_\omega(i(\sigma)) = \int_{\st(i(\sigma))} \omega\Big),
\end{align*}
on the other hand, 
\begin{align*}
\psi_N \circ i^*(\omega) &= \psi_N(i^*(\omega)) \\
&= \Big(\sigma \mapsto c_{i^*(\omega)}(\sigma) = \int_{\st(\sigma)} i^*\omega
\Big) \\
&=\Big(\sigma \mapsto  \int_{i(\st(\sigma))} \omega \Big).
\end{align*}
Since $i(\st(\sigma)) = \st(i(\sigma))$, we get that the two cocycles coincide. This shows that both the diagrams commute, whence we get the thesis.
\end{proof}
Theorem~\ref{thm:totally:geodesic} is then an easy consequence of Proposition~\ref{prop:totally:geodesic:submfld}:
\begin{proof}[Proof of Theorem~\ref{thm:totally:geodesic}]
    It is sufficient to notice that when $M$ is an oriented compact connected negatively curved manifold of dimension $\dim(M) \geq 3$ with convex boundary and totally geodesic boundary component $M_0$, then $M_0$ is an oriented closed connected negatively curved manifold.
    Hence, we can apply Proposition~\ref{prop:totally:geodesic:submfld} by setting $N = M_0$.
\end{proof}

\section{Non-injective examples}\label{sec:noninj}
One can easily construct simple examples of negatively curved manifolds with non-empty convex
boundary such that the Barge--Ghys straightening morphism $\Psi$ is trivial (and so non-injective).

\begin{example}
  Let $n \geq 2$ and let $M$ be an oriented compact connected negatively curved $n$-manifold $M$, with non-empty
  convex boundary. In this case $C\Omega^2(M)$ is not trivial. Suppose moreover that $M$ is homotopy
  equivalent to a space $X$ with $H_b^2(X;\R)=0$. Since the bounded cohomology is a homotopy
  invariant~\cite{vbc, Ivanov17}, then $H_b^2(M;\R)=0$ and so the Barge--Ghys straightening morphism
  $\Psi$ is trivial.  
\end{example}

 Since the point and $S^1$ have trivial second bounded cohomology~\cite{vbc, Ivanov17}, as examples
 of the above classes one can take a ball in a hyperbolic space, or a convex tubular
 neighbourhood of a simple closed geodesic in a hyperbolic manifold.

 
\bibliographystyle{alpha}
\bibliography{Biblio}

\end{document}